\documentclass[article,12pt]{amsart}

\usepackage{amsfonts}
\usepackage{amsmath}
\usepackage{mathrsfs}
\usepackage{graphicx}
\usepackage{color}
\usepackage{epsfig}
\usepackage{yfonts}
\usepackage{fancybox}
\usepackage{enumerate}
\usepackage{dsfont}
\usepackage{etoolbox}
\apptocmd{\thebibliography}{}{}{}

\numberwithin{equation}{section}
\theoremstyle{plain}
\newtheorem{thm}{Theorem}[section]
\newtheorem{rem}{Remark}[section]

\newtheorem{lem}{Lemma}[section]
\newtheorem{deff}{Definition}[section]

\newcommand{\dE}{\mathbb{E}}
\newcommand{\dR}{\mathbb{R}}
\newcommand{\dL}{\mathbb{L}}

\newcommand{\dP}{\mathbb{P}}
\newcommand{\dZ}{\mathbb{Z}}
\newcommand{\dC}{\mathbb{C}}

\newcommand{\cL}{\mathcal{L}}
\newcommand{\cN}{\mathcal{N}}

\newcommand{\rI}{\mathrm{I}}
\newcommand{\cF}{\mathcal{F}}

\newcommand{\cM}{\mathcal{M}}
\newcommand{\veps}{\varepsilon}

\newcommand{\ind}{\mbox{1}\kern-.25em \mbox{I}}
\font\calcal=cmsy10 scaled\magstep1
\def\build#1_#2^#3{\mathrel{\mathop{\kern 0pt#1}\limits_{#2}^{#3}}}
\def\liml{\build{\longrightarrow}_{}^{{\mbox{\calcal L}}}}

\def\videbox{\mathbin{\vbox{\hrule\hbox{\vrule height1.4ex \kern.6em\vrule height1.4ex}\hrule}}}
\def\demend{\hfill $\videbox$\\}
\email{bernard.bercu@math.u-bordeaux.fr}


\begin{document}
\title[On the elephant random walk with stops]
{On the multidimensional elephant random walk with stops \vspace{1ex}}
\author{Bernard Bercu}
\dedicatory{\normalsize University of Bordeaux, France}
\address{Universit\'e de Bordeaux, Institut de Math\'ematiques de Bordeaux,
UMR CNRS 5251, 351 Cours de la Lib\'eration, 33405 Talence cedex, France.}
\thanks{}

\begin{abstract}
The goal of this paper is to investigate the asymptotic behavior of 
the multidimensional elephant random walk with stops (MERWS). 
In contrast with the standard elephant random walk, the elephant is allowed to stay on his own position. We prove that the Gram matrix associated with the MERWS, properly
normalized, converges almost surely to the product of a deterministic matrix, related to the axes 
on which the MERWS moves uniformly, and a Mittag-Leffler distribution.
It allows us to extend all the results previously established for the one-dimensional elephant random walk with stops. More precisely, in the diffusive and critical regimes, we prove the almost sure convergence
of the MERWS. In the superdiffusive regime, we establish the almost sure convergence of
the MERWS, properly normalized, to a nondegenerate random vector. We also study the self-normalized asymptotic normality of the MERWS.
\end{abstract}
\maketitle


\ \vspace{-5ex}
\section{Introduction}
\label{S-I}

Over the last decade, the elephant random walk (ERW) has attracted a growing attention 
in mathematics and statistical physics \cite{Baur2016, Bercu2018, Bertoin2021,
Bertoin2022, Dedecker2023, Fan2021, Hu2025, Kiss2022, Kubota2019, Schutz2004}. 
While a wide range of litterature is now available on the
multidimensional elephant random walk \cite{BercuLaulin2019, BercuLaulin2021, Bertenghi2022, Chen2023, Gonzalez2020},
no result can be found on the multidimensional elephant random walk with stops (MERWS), which
is the natural extension to higher dimension of the one-dimensional ERW with stops.
The goal of this paper is to fill the gap by extending the recent results in \cite{Bercu2022} 
to the multidimensional setting. One can clearly see below that this extension is far from being simple as it involves the product of a deterministic matrix and the Mittag-Leffler distribution. 
We refer the reader to the recent contribution \cite{Roy2025} on the ERW with stops in a triangular array setting
where the Mittag-Leffler distribution also plays a crucial role.
\vspace{1ex} \\
\noindent
For a given dimension $d \geq 1$, let $(S_n)$ be a random walk on $\dZ^d$, 
starting at the origin at time zero, $S_0=0$.
For the first step, the elephant moves in one of the $2d$ directions of on $\dZ^d$ with the same probability $1/2d$. The next steps are performed as follows. Choose at random an integer $k$
among the previous times $\{1,\ldots,n\}$. Then, the elephant moves exactly in the same direction as at time $k$ with probability $p$, 
or to one of the $2d-1$ remaining directions with the same probability $q$, or it stays to his own position with probability $r$. In other words, for all $n \geq 1$,
\begin{equation}
\label{INCMERW}
X_{n+1}=A_{n+1} X_k
\end{equation}
with
\begin{equation}
\label{DEFMATRIXA}
   A_{n+1} = \left \{ \begin{array}{ccc}
    +I_d &\text{ with probability } & p \vspace{1ex}\\
    -I_d &\text{ with probability } & q \vspace{1ex}\\
     +J_d &\text{ with probability } & q \vspace{1ex}\\
    -J_d &\text{ with probability } & q \vspace{1ex}\\
     & \vdots & \vspace{1ex}\\
      +J_d^{d-1} &\text{ with probability } & q \vspace{1ex}\\
    -J_d^{d-1}  &\text{ with probability } & q \vspace{1ex}\\
    0  &\text{ with probability } & r 
   \end{array}  \right.
   \vspace{2ex}
\end{equation}
where $I_d$ and $J_d$ are the square matrices of order $d$ defined by
\begin{equation}
\label{DEFIJ}
I_d =
\begin{pmatrix}
   1 & 0 & \cdots &  \cdots & 0 \\
   0 & 1 & 0 & \cdots & 0 \\
    \vdots & \ddots & \ddots & \ddots & \vdots \\
  0& \cdots & 0& 1 & 0 \\
   0 & \cdots & \cdots & 0 & 1 
\end{pmatrix}
\hspace{1cm}
\text{and}
\hspace{1cm}
J_d = 
\begin{pmatrix}
   0 & 1 &  0 & \cdots & 0 \\
   0 & 0 & 1  & \ddots & \vdots \\
   \vdots & \ddots &  \ddots & \ddots & 0 \\
   0 & \cdots & 0 & 0 & 1 \\
   1 & 0 & \cdots& 0 & 0 
\end{pmatrix}
\end{equation}
and where
$$
p+(2d-1)q+r=1.
$$
Therefore, the position of the MERWS at time $n+1$ is given by
\begin{equation} 
\label{POSMERWS}
S_{n+1} = S_n + X_{n+1}. 
\end{equation}
In all the sequel, we assume that $0 < r < 1$ inasmuch as the
case $r = 0$ was previously investigated by Bercu and Laulin \cite{BercuLaulin2019}, while in the case $r=1$, the elephant remains stuck at the origin after the first step. 
It follows from our definition of the MERWS that for any $n \geq 1$,
\begin{equation}
\label{STEPSMERWS}
   X_{n+1} = A_{n+1} X_{U_{n+1}}
\end{equation}
where $U_{n+1}$ stands for a random variable uniformly distributed on $\{1, \ldots,n\}$. Moreover,
$A_n$ and $U_{n+1}$ are conditionnaly independent given $\cF_n$ where $\cF_n=\sigma(X_1,\ldots,X_n)$.
Consequently, we obtain from \eqref{STEPSMERWS} and the law of total probability that for all $n \geq 1$,
\begin{align*}
   \dE[X_{n+1}| \cF_n] &= \dE[A_{n+1} X_{U_{n+1}}| \cF_n] = \sum_{k=1}^n  \dE[A_{n+1} X_k\rI_{\{U_{n+1}=k\}}| \cF_n] \hspace{1cm} \text{a.s.} \\
   &= \sum_{k=1}^n  \dE[A_{n+1}] \dP(U_{n+1}=k) X_k=\frac{1}{n}  \sum_{k=1}^n  \dE[A_{n+1}] X_k 
   \hspace{1cm} \text{a.s.}
\end{align*}
which implies via \eqref{DEFMATRIXA} that
\begin{equation}
\label{CEMERWS}
   \dE[X_{n+1}| \cF_n] = \frac{a}{n}S_n
   \hspace{1cm} \text{a.s.}
\end{equation}
where the fundamental parameter $a$ of the MERWS is given by
\begin{equation}
\label{DEFA}
a=p-q=\frac{2dp+r-1}{2d-1}.
\end{equation}
Hence, we obtain from \eqref{POSMERWS} and \eqref{CEMERWS} that almost surely
\begin{equation}
\label{CEPOSMERWS}
   \dE[S_{n+1}| \cF_n] = \alpha_nS_n
   \hspace{1cm} \text{where} \hspace{1cm} \alpha_n=1+\frac{a}{n}.
\end{equation}
We shall see in Section \ref{S-MA} below that the critical value associated with the memory parameter $p$ of MERWS is given by
\begin{equation}
\label{DEFPDR}
p_{d,r}=\left(\frac{2d+1}{4d}\right)(1-r).
\end{equation}
One can observe that the standard ERW without stops \cite{Baur2016, Bercu2018, Coletti2017, Kubota2019, Schutz2004} corresponds to $d=1$ and $r=0$, which implies that $p_{d,r}$ reduces to the well-known critical value $3/4$.
The MERWS is said to be diffusive if $p<p_{d,r}$, critical if $p=p_{d,r}$ and
superdiffusive if $p>p_{d,r}$. A crucial point in our analysis is the study of the asymptotic behavior of
the Gram matrix $\Sigma_n$ associated with the MERWS, defined by
\begin{equation} 
\label{GRAMMERWS}
\Sigma_{n} = \sum_{k=1}^n X_k X_k^T. 
\end{equation}
We shall improve Lemma 2.1 in \cite{Bercu2022} as well as Theorem 3.1 in \cite{Gutb2021} 
by showing that, whatever the values of the $p, q \in [0,1]$ and $r \in ]0,1[$
\begin{equation} 
\label{LIMGRAMMERWS}
\lim_{n \rightarrow \infty} \frac{1}{n^{1-r}}\Sigma_{n} = \frac{1}{d}\Sigma I_d  \hspace{1cm} \text{a.s.}
\end{equation}
where 
$\Sigma$ stands for a Mittag-Leffler distribution with parameter $1-r$. The matrix $d^{-1}I_d$ is related to the axes 
of $\dZ^d$ on which the MERWS moves uniformly.
By taking the trace on both sides of \eqref{LIMGRAMMERWS}, if we denote $\sigma_n^2=\text{Tr}(\Sigma_n)$, we clearly obtain that
\begin{equation} 
\label{LIMTRGRAMMERWS}
\lim_{n \rightarrow \infty} \frac{\sigma_n^2}{n^{1-r}}  = \Sigma   \hspace{1cm} \text{a.s.}
\end{equation}
The almost sure convergences \eqref{LIMGRAMMERWS} and \eqref{LIMTRGRAMMERWS} will allow us to carry out
a sharp analysis of the asymptotic behavior of the MERWS. 
The paper is organized as follows. Section \ref{S-ML} deals with the asymptotic behavior of the Gram matrix $\Sigma_n$.
Section \ref{S-MR} is devoted to the main results of the paper. We establish the almost sure asymptotic behavior of the
MERWS in the diffusive, critical and superdiffusive regimes. Moreover, we also prove the asymptotic normality of the MERWS, suitably normalized by
$\sigma_n^2$, in the diffusive and critical regimes. Finally, the fluctuation of the MERWS around its limiting random variable is also provided in the superdiffusive regime. 
Our multidimensional martingale approach is described  in Section \ref{S-MA}, while all technical proofs are 
postponed to Appendices A to C.

\section{On the asymptotic behavior of the Gram matrix}
\label{S-ML}

The Mittag-Leffler function was introduced in 1903 by the Swedish mathematician G\"osta Mittag-Leffler. It has a rich and long history in complex analysis and probability.
It is defined, for all $z \in \dC$, by
$$
E_\alpha(z) = \sum_{n=0}^\infty \frac{z^n}{\Gamma(1+n\alpha )}
$$
where $\alpha$ is a positive real parameter and $\Gamma$ stands for the Euler Gamma function. We refer the reader to \cite{Gorenflo2020} for a monograph devoted to the main properties of Mittag-Leffler functions.

\begin{deff}
\label{DEFML}
We shall say that a positive random variable $X$ has a Mittag-Leffler distribution with parameter $\alpha \in [0,1]$, denoted by $\cM\cL(\alpha)$, if its Laplace transform is given, for all $t\in \dR$, by
\begin{equation}
\label{LAPML}
\dE[\exp(tX)]=E_\alpha(t) = \sum_{n=0}^\infty \frac{t^n}{\Gamma(1+n\alpha )}.
\end{equation}
\end{deff}
\noindent
If $X$ has a $\cM\cL(\alpha)$ distribution with parameter $0<\alpha<1$, its probability density function $f_\alpha$ is given \cite{Feller1971}, for all $x>0$, by
\begin{equation}
\label{PDFML}
f_\alpha(x)=\frac{1}{\pi \alpha} \sum_{n=0}^\infty \Gamma(1+\alpha n)\sin(\alpha n \pi) \frac{(-x)^{n-1}}{n!}.
\end{equation}
As a special case, we have for all $x>0$
$$
f_{1/2}(x)=\frac{1}{\sqrt{\pi}} \exp\left(\! -\frac{x^2}{4}\right).
$$ 
It means that the $\cM\cL(\alpha)$ distribution with parameter $\alpha=1/2$ coincides with the distribution of 
$|Z|$ where $Z$ has a Gaussian $\cN(0,2)$ distribution. The Mittag-Leffler distribution satisfies the celebrated Carleman's condition which means that it is characterized by its moments. They are given, for any integer $m \geq 1$, by
\begin{equation}
\label{DEFMLMOM}
\dE[X^m]=\frac{m!}{\Gamma(1+m \alpha )}.
\end{equation}
\noindent
Our first result on the Gram matrix associated with the MERWS is as follows.

\begin{lem}
\label{L-ML}
Whatever the values of $p$, $q$ in $[0,1]$ and $r$ in $]0,1[$, we have 
\begin{equation}
\label{ASCVGML}
\lim_{n \rightarrow \infty} \frac{1}{n^{1-r}}\Sigma_n = \frac{1}{d}\Sigma I_d
\hspace{1cm}\text{a.s.}
\end{equation}
where  $\Sigma$ stands for the $\cM\cL(1-r)$ distribution. In particular,
\begin{equation}
\label{ASCVGMLTR}
\lim_{n \rightarrow \infty} \frac{\sigma_n^2}{n^{1-r}}= \Sigma 
\hspace{1cm}\text{a.s.}
\end{equation}
Moreover, this convergence holds in $\dL^m$ for any integer $m\geq 1$,
\begin{equation}
\label{MLMOM}
\lim_{n \rightarrow \infty} \dE \left[ \left| \frac{\sigma_n^2}{n^{1-r}} - \Sigma \right|^m \right]=0.
\end{equation}
\end{lem}

\begin{proof}
The proof is given in Section \ref{S-MA}. 
\end{proof}



\ \vspace{-2ex}
\section{Main results}
\label{S-MR}


\subsection{The diffusive regime}

First of all, we focus our attention on the asymptotic behavior of the MERWS
in the diffusive regime where $p<p_{d,r}$. The almost sure convergence of the position of
the MERWS is as follows.

\begin{thm}
\label{T-ASCVG-DR}
We have the almost sure convergence
\begin{equation}
\label{ASCVGDR}
\lim_{n \rightarrow \infty} \frac{1}{n}S_n = 0 \hspace{1cm}\text{a.s.}
\end{equation}
More precisely,
\begin{equation}
\label{ASCVGDRSHARP}
\frac{\|S_n\|^2}{n^2}=O \left( \frac{\log n}{n^{1+r}} \right)
\hspace{1cm} \text{a.s.}
\end{equation}
\end{thm}

\noindent
Our second result is devoted to the law of the iterated logarithm for the MERWS.  
Denote by $v^2$ the asymptotic variance
\begin{equation}
\label{VARDR}
v^2=\frac{(2d-1)(1-r)}{d((2d+1)(1-r)-4dp)}.
\end{equation}

\begin{thm}
\label{T-LIL-DR}
We have the law of the iterated logarithm
\begin{equation}
 \limsup_{n \rightarrow \infty} \frac{\|S_n\|^2}{2 \sigma_n^2 \log \log \sigma_n^2}=  
d v^2 \hspace{1cm} \text{a.s.}
\label{LIL-DR1}
\end{equation}
Moreover, we also have 
\begin{equation}
\limsup_{n \rightarrow \infty} \frac{\|S_n\|^2}{2 n^{1-r} \log \log n}=  
d v^2 \, \Sigma \hspace{1cm} \text{a.s.}
\label{LIL-DR2}
\end{equation}
where $\Sigma$ has a $\cM\cL(1-r)$ distribution. 
\end{thm}

\begin{rem} The law of the iterated logarithm \eqref{LIL-DR2} clearly improves \eqref{ASCVGDRSHARP} as
\begin{equation*}
\frac{\|S_n\|^2}{n^2}=O \left( \frac{\log \log n}{n^{1+r}} \right)
\hspace{1cm} \text{a.s.}
\end{equation*}
\end{rem}

\noindent
Our next result deals with the asymptotic normality of the MERWS. As previously seen, it is necessary
to self-normalized the position $S_n$ by the trace $\sigma_n^2$ of the Gram matrix $\Sigma_n$ in order 
to establish the asymptotic normality. 

\begin{thm}
\label{T-AN-DR}
We have the asymptotic normality
\begin{equation}
\label{ANDR}
\frac{S_n}{\sqrt{\sigma_n^2}} \underset{n\rightarrow+\infty}{\overset{\cL}{\longrightarrow}} 
\cN\big(0,v^2 I_d\big).
\end{equation}
Moreover, we also have
\begin{equation}
\label{ANDRML}
\frac{S_n}{\sqrt{n^{1-r}}} \underset{n\rightarrow+\infty}{\overset{\cL}{\longrightarrow}} \sqrt{\Sigma^\prime}
\cN\big(0,v^2 I_d\big)
\end{equation}
where $\Sigma^\prime$ is independent of the Gaussian distribution at the right-hand side of \eqref{ANDRML}
and $\Sigma^\prime$ has a $\cM\cL(1-r)$ distribution.
\end{thm}

\begin{rem}
In the special case $d=1$, we find again the law of iterated logarithm and the asymptotic normality
given by Theorem 3.2 and Theorem 3.3 in \cite{Bercu2022}.
\end{rem}


\subsection{The critical regime}

We now turn to the asymptotic behavior of the MERWS in the critical
regime where $p=p_{d,r}$. We start with the almost sure convergence of the position of
the MERWS.

\begin{thm}
\label{T-ASCVG-CR}
We have the almost sure convergence
\begin{equation}
\label{ASCVGCR}
\lim_{n \rightarrow \infty} \frac{1}{n}S_n = 0 \hspace{1cm}\text{a.s.}
\end{equation}
More precisely,
\begin{equation}
\label{ASCVGCRSHARP}
\frac{\|S_n\|^2}{n^2}=O \left( \frac{\log n \log \log n}{n^{1+r}} \right)
\hspace{1cm} \text{a.s.}
\end{equation}
\end{thm}

\noindent
Hereafter, we focus our attention on the law of iterated logarithm for the MERWS.

\begin{thm}
\label{T-LIL-CR}
We have the law of the iterated logarithm
\begin{equation}
 \limsup_{n \rightarrow \infty} \frac{\|S_n\|^2}{2 \sigma_n^2 \log \sigma_n^2 \log \log \log \sigma_n^2}=  
1\hspace{1cm} \text{a.s.}
 \label{LIL-CR1}
\end{equation}
Moreover, we also have 
\begin{equation}
 \limsup_{n \rightarrow \infty} \frac{\|S_n\|^2}{2 n^{1-r} \log n \log \log \log n}=  
 (1-r) \Sigma \hspace{1cm} \text{a.s.}
 \label{LIL-CR2}
\end{equation}
where $\Sigma$ has a $\cM\cL(1-r)$ distribution.
\end{thm}

\begin{rem} The law of iterated logarithm
\eqref{LIL-CR2} improves \eqref{ASCVGCRSHARP} as
\begin{equation*}
\label{RATECR}
\frac{\|S_n\|^2}{n^2} = O\left( \frac{\log n \log \log \log n}{n^{1+r}} \right) \hspace{1cm}\text{a.s.}
\end{equation*}
\end{rem}

\noindent
Our next result concerns the asymptotic normality for the MERWS.

\begin{thm}
\label{T-AN-CR}
We have the asymptotic normality
\begin{equation}
\label{ANCR}
\frac{S_n}{\sqrt{\sigma_n^2 \log \sigma_n^2 }} \underset{n\rightarrow+\infty}{\overset{\cL}{\longrightarrow}} \cN\Big(0,\frac{1}{d} I_d\Big).
\end{equation}
Moreover, we also have
\begin{equation}
\label{ANCRML}
\frac{S_n}{\sqrt{n^{1-r} \log n}} \underset{n\rightarrow+\infty}{\overset{\cL}{\longrightarrow}} \sqrt{(1-r)\Sigma^\prime}\cN\Big(0,\frac{1}{d} I_d\Big)
\end{equation}
where $\Sigma^\prime$ is independent of the Gaussian distribution at the right-hand side of \eqref{ANCRML}
and $\Sigma^\prime$ has a $\cM\cL(1-r)$ distribution.
\end{thm}

\begin{rem}
In the special case $d=1$, we find again the law of iterated logarithm and the asymptotic normality
given by Theorem 3.5 and Theorem 3.6 in \cite{Bercu2022}.
\end{rem}


\subsection{The superdiffusive regime}

Finally, we investigate the asymptotic behavior of the MERWS in the superdiffusive
regime where $p>p_{d,r}$. Denote by $\vartheta^2$ the asymptotic variance
\begin{equation}
\label{VARSR}
\vartheta^2=\frac{(2d-1)(1-r)}{d(4dp-(2d+1)(1-r))}.
\end{equation}

\begin{thm}
\label{T-ASCVG-SR}
We have the almost sure convergence
\begin{equation}
\label{ASCVGSR1}
\lim_{n \rightarrow \infty} \frac{S_n}{n^{p-q}} = L \hspace{1cm}\text{a.s.}
\end{equation}
where $L$ is a non-degenerate random vector. 
Moreover, we also have the mean square convergence
\begin{equation}
\label{ASCVGSR2}
 \lim_{n \rightarrow \infty} \dE\Bigl[ \Bigl\| \frac{S_n}{n^{p-q}} -L \Bigr\|^2 \Bigr]=0.
\end{equation}
In addition, $\dE[L]=0$ and its covariance matrix is given by
\begin{equation}
\label{COVL}
\dE[L L^T]=\frac{\vartheta^2}{(1-r)\Gamma(2(p-q))}I_d.
\end{equation}
\end{thm}

\noindent
Our last result is devoted to the Gaussian fluctuations of the MERWS around $L$, 
in the spirit of the seminal work of Kubota and Takei \cite{Kubota2019} inspired by Heyde \cite{Heyde1977}. 

\begin{thm}
\label{T-AN-SR}
We have the Gaussian fluctuation
\begin{equation}
\label{ANSR1}
\frac{S_n-n^{p-q}L}{\sqrt{\sigma_n^2}}
\liml \cN\big(0,\vartheta^2 I_d\big).
\end{equation}
Moreover, we also have
\begin{equation}
\label{ANSR2}
\frac{S_n -n^{p-q}L }{\sqrt{n^{1-r}}} 
\liml \sqrt{\Sigma^\prime}\cN\big(0,\vartheta^2I_d\big)
\end{equation}
where $\Sigma^\prime$ is independent of the Gaussian distribution at the right-hand side of \eqref{ANSR2}
and $\Sigma^\prime$ has a $\cM\cL(1-r)$ distribution.
\end{thm}

\begin{rem}
In the special case $d=1$, we find again the variance of $L$ as well as the Gaussian fluctuations given by 
Theorem 3.8 and Theorem 3.9 in \cite{Bercu2022}.  
\end{rem}

\vspace{-2ex}
\section{Our multidimensional martingale approach}
\label{S-MA}
Let $(M_n)$ be the sequence of $\dR^d$ defined, for all $n \geq 0$, by 
\begin{equation}
\label{DEFMN}
M_n=a_nS_n
\end{equation}
where the initial values $a_0=1$, $a_1=1$ and for all $n \geq 2$,
\begin{equation}
\label{DEFAN}
a_n=\frac{\Gamma(n)\Gamma(a+1)}{\Gamma(n+a)}= \frac{ a(n-1)!}{(a)^{(n)}}
\end{equation}
where $(a)^{(n)}$ stands for the rising factorial defined by $(a)^{(n)}=a(a+1) \cdots(a+n-1)$.
It follows from the asymptotic behavior of the Euler Gamma function that
$$
\lim_{n \rightarrow \infty} \frac{\Gamma(n+a)}{\Gamma(n)n^a}=1,
$$
which immediately leads to
\begin{equation}
\label{LIMAN}
\lim_{n \rightarrow \infty} n^a a_n=\Gamma(a+1).
\end{equation}
Moreover, since $a_n=a_{n+1} \alpha_n$, we obtain from
\eqref{CEPOSMERWS} and \eqref{DEFMN} that 
$$
\dE[M_{n+1} | \cF_n]=a_{n+1}\dE[S_{n+1} | \cF_n]=a_{n+1} \alpha_n S_n=a_{n}S_n=M_n
\hspace{1cm} \text{a.s.}
$$
Hence, $(M_n)$ is a multidimensional discrete martingale sequence which can be rewritten in 
the additive form
\begin{equation} 
\label{NEWDEFMN}
M_n=\sum_{k=1}^n a_k \varepsilon_k
\end{equation}
where the martingale difference sequence $(\varepsilon_n)$ is given by $\varepsilon_1=S_1$ and, for all $n \geq 1$,
$\varepsilon_{n+1}=S_{n+1}-\alpha_{n} S_{n}$. The predictable quadratic variation associated with $(M_n)$ is the random square matrix of order $d$ given, for all $n \geq 1$, by
\begin{equation} 
\label{PQVMN}
\langle M \rangle_n =\frac{1}{d} I_d+ \sum_{k=1}^{n-1} a_{k+1}^2\dE[\varepsilon_{k+1} \varepsilon_{k+1}^T | \mathcal{F}_{k}]
\hspace{1cm} \text{a.s.}
\end{equation}
In addition, through a careful use of \eqref{DEFMATRIXA} and \eqref{STEPSMERWS}, we have that
\begin{align*}
   \dE[X_{n+1}X_{n+1}^T| \cF_n] &= \dE[A_{n+1} X_{U_{n+1}} X_{U_{n+1}}^T A_{n+1}^T | \cF_n]  \hspace{1cm} \text{a.s.} \\
   &= \sum_{k=1}^n  \dE[A_{n+1} X_k X_k^T A_{n+1}^T \rI_{\{U_{n+1}=k\}}| \cF_n] \hspace{1cm} \text{a.s.} \\
   &=  \frac{1}{n}  \sum_{k=1}^n  \dE[A_{n+1} X_k X_k^T A_{n+1}^T | \cF_n] \hspace{1cm} \text{a.s.} \\
   &=  \frac{(p+q)}{n}  \sum_{k=1}^n   X_k X_k^T  +  \frac{2q}{n}  \sum_{k=1}^n \sum_{i=1}^{d-1} J_d^i X_k X_k^T (J_d^i)^T   \hspace{1cm} \text{a.s.}\\
   &=  \frac{(p-q)}{n}  \sum_{k=1}^n   X_k X_k^T  +  \frac{2q}{n}  \sum_{k=1}^n \sum_{i=1}^{d} J_d^i X_k X_k^T (J_d^i)^T 
\hspace{1cm} \text{a.s.}
\end{align*}
which implies via the definition of the permutation matrix $J_d$ given by \eqref{DEFIJ} that
\begin{equation}
\label{CVARMERWS}
   \dE[X_{n+1}X_{n+1}^T| \cF_n] = \frac{a}{n}\Sigma_n + \frac{2q}{n}\sigma_n^2 I_d
   \hspace{1cm} \text{a.s.}
\end{equation}
where
\begin{equation}
\label{TRGRAMMERWS}
\sigma_{n}^2 = \text{Tr}(\Sigma_n) = \sum_{k=1}^n \| X_k \|^2.
\end{equation}
Therefore, we obtain from \eqref{CEMERWS} and \eqref{CVARMERWS} that
\begin{align}
   \dE[S_{n+1}S_{n+1}^T| \cF_n] &= \dE[(S_{n}+X_{n+1})(S_n+X_{n+1})^T | \cF_n]  \hspace{1cm} \text{a.s.} \nonumber\\
   &=\left(1+\frac{2a}{n}\right) S_nS_n^T + \frac{a}{n}\Sigma_n + \frac{2q}{n}\sigma_n^2 I_d
\hspace{1cm} \text{a.s.}
\label{CVARSN}
\end{align}
leading to
\begin{align}
   \dE[\varepsilon_{n+1}\varepsilon_{n+1}^T| \cF_n] &=  \dE[S_{n+1}S_{n+1}^T| \cF_n] - \alpha_n^2 S_n S_n^T
     \hspace{1cm} \text{a.s.} \nonumber\\
   &=\left(1+\frac{2a}{n} - \alpha_n^2 \right) S_nS_n^T + \frac{a}{n}\Sigma_n + \frac{2q}{n}\sigma_n^2 I_d
\hspace{1cm} \text{a.s.} \nonumber\\
  &=  \frac{a}{n}\Sigma_n + \frac{2q}{n}\sigma_n^2 I_d - \left(\frac{a}{n} \right)^2 S_nS_n^T
\hspace{1cm} \text{a.s.}
\label{CVAREPSILON}
\end{align}
Consequently, we find from \eqref{PQVMN} and \eqref{CVAREPSILON} that the predictable quadratic variation
$\langle M \rangle_n$ can be splitted into three matrices
\begin{equation}
\label{CALCIPM}
\langle M \rangle_n = \frac{1}{d}I_d + V_n - a^2 W_n
\hspace{1cm} \text{a.s.}
\end{equation}
where
\begin{equation*}
V_n = a\sum_{k=1}^{n-1}  \left(\frac{a_{k+1}^2}{k}\right) \Sigma_k
+2q \sum_{k=1}^{n-1}  \left(\frac{a_{k+1}^2}{k}\right) \sigma_k^2 I_d
\qquad \text{and} \qquad
W_n = \sum_{k=1}^{n-1}  \left(\frac{a_{k+1}}{k}\right)^2 S_kS_k^T.
\end{equation*}
As it was already the case for the unidimensional ERWS \cite{Bercu2022}, the main difficulty arising here is that the Gram matrix $\Sigma_n$, properly normalized, will converge to the product of a Mittag-Leffler random variable and a deterministic matrix. We deduce from \eqref{GRAMMERWS} and \eqref{CVARMERWS} that for all $n \geq 1$,
\begin{equation}
\label{CEGRAM}
   \dE[\Sigma_{n+1}| \cF_n] = \left(1+\frac{a}{n}\right)\Sigma_n + \frac{2q}{n}\sigma_n^2 I_d
   \hspace{1cm} \text{a.s.}
\end{equation}
By taking the trace on both sides of \eqref{CEGRAM}, we obtain from \eqref{TRGRAMMERWS} that for all $n \geq 1$,
\begin{equation}
\label{CETRGRAM}
   \dE[\sigma_{n+1}^2| \cF_n] = \left(1+\frac{b}{n}\right)\sigma_n^2 
   \hspace{1cm} \text{a.s.}
\end{equation}
where the second fundamental parameter $b$ is given by
\begin{equation}
\label{DEFB}
b=a+2dq=1-r.
\end{equation}
One can observe $\Sigma_n$ is the diagonal matrix of order $d$ given by
\begin{equation}
\label{GRAMDIAG}
\Sigma_n = 
\begin{pmatrix}
   \sigma_n^2(1) & 0 & \cdots &   0 \\
   0 & \sigma_n^2(2) & \ddots &  \vdots \\
    \vdots & \ddots & \ddots &  0 \\
   0 & \cdots & 0 & \sigma_n^2(d)
\end{pmatrix}
\end{equation}
where for all $1\leq i \leq d$,
$$
 \sigma_n^2(i)=\sum_{k=1}^n X_k^2(i)
$$
and $X_n(i)$ stands for the $i$-th coordinate of the random vector $X_n$. In addition, one can immediately see that
\begin{equation}
\label{DECSIGMAN2}
\sigma_n^2=\sum_{i=1}^d \sigma_n^2(i)= \sum_{i=1}^d \sum_{k=1}^n X_k^2(i).
\end{equation}
The asymptotic behavior of $\sigma_n^2$ will allow us to identify the asymptotic behavior of $(\sigma_n^2(1),\ldots,\sigma_n^2(d))$, that is the one of the Gram matrix $\Sigma_n$. Hereafter, let $(N_n)$ be the sequence defined, for all $n \geq 1$, by 
\begin{equation}
\label{DEFNN}
N_n=b_n \sigma_n^2
\end{equation}
where the initial term $b_1=1$ and for all $n \geq 2$,
\begin{equation}
\label{DEFBN}
b_n=\frac{\Gamma(n)\Gamma(b+1)}{\Gamma(n+b)}= \frac{ b(n-1)!}{(b)^{(n)}}.
\end{equation}
As previously seen for the multidimensional martingale $(M_n)$, it follows from \eqref{CETRGRAM} and \eqref{DEFNN} that 
$\dE[N_{n+1} | \cF_n]=b_{n+1}\dE[\sigma_{n+1}^2 | \cF_n]=b_{n}\sigma_n^2=N_n$ a.s. which means that $(N_n)$ is
a unidimensional discrete martingale sequence. Its asymptotic behavior was previously established in 
Lemma 4.1 of \cite{Bercu2022} as follows.

\begin{lem}
\label{L-MARTN}
The martingale $(N_n)$ is bounded in $\dL^m$ for any integer $m\geq 1$. 
More precisely, for all $n\geq 1$ and for any integer $m\geq 1$, 
\begin{equation}
\label{MOMN}
\dE[N_n^m] \leq m!.
\end{equation}
Consequently, $(N_n)$ converges almost surely and in $\dL^m$ to a finite random variable $N$ satisfying
for any integer $m\geq 1$, 
\begin{equation}
\label{CALCMOMN}
\dE[N^m] =\frac{m!(\Gamma(b+1))^m}{\Gamma(1+mb)}.
\end{equation}
\end{lem}

\noindent
We are now in position to prove Lemma \ref{L-ML}.
\vspace{1ex} \\
{\bf Proof of Lemma \ref{L-ML}.} According to \eqref{DEFBN}, we have
\begin{equation}
\label{LIMBN}
\lim_{n \rightarrow \infty} n^b b_n=\Gamma(b+1).
\end{equation}
Consequently, we deduce from Lemma \ref{L-MARTN} together with  \eqref{DEFNN}
and \eqref{LIMBN} that it exists a finite random variable 
$$\Sigma=\frac{N}{\Gamma(b+1)}$$ 
such that
\begin{equation}
\label{LIMSIGMAN2}
\lim_{n \rightarrow \infty} \frac{1}{n^b} \sigma_n^2 =\Sigma \hspace{1cm} \text{a.s.}
\end{equation}
Moreover, we have from \eqref{CALCMOMN} and \eqref{LIMBN} that 
for all integer $m\geq 1$,
$$
\dE[\Sigma^m]=\frac{\dE[N^m]}{( \Gamma(b+1))^m}=\frac{m!}{\Gamma(1+mb)}.
$$
The moments of $\Sigma$ are those of the $\cM \cL(b)$ distribution given by 
\eqref{DEFMLMOM}. As the Mittag-Leffler distribution is characterized by its moments, it means
that the random variable $\Sigma$ has a $\cM \cL(b)$ distribution. Therefore, as $b=1-r$, 
convergence \eqref{LIMSIGMAN2} immediately leads to \eqref{ASCVGMLTR}. Moreover, \eqref{MLMOM} follows from 
\eqref{LIMBN} and together with the fact that $(N_n)$ converges in $\dL^m$ for any integer $m\geq 1$. 
We now carry on with the proof of convergence \eqref{ASCVGML} which is much more
difficult to establish.
We have from \eqref{CEGRAM} that for all $1 \leq i \leq d$ and for all $n \geq 1$,
\begin{align}
  \sigma_{n+1}^2(i) & =  \sigma_{n+1}^2(i) - \dE[\sigma_{n+1}^2(i)| \cF_n] +\dE[\sigma_{n+1}^2(i)| \cF_n] \hspace{1cm} \text{a.s.} \nonumber \\
  &=\left(1+\frac{a}{n}\right)\sigma_n^2(i)+\frac{2q}{n}\sigma_{n}^2 + \sigma_{n+1}^2(i) - \dE[\sigma_{n+1}^2(i)| \cF_n] \hspace{1cm} \text{a.s.} \nonumber \\
  &=\alpha_n\sigma_n^2(i)+\frac{2q}{n}\sigma_{n}^2 + \xi_{n+1}(i) 
   \hspace{1cm} \text{a.s.}
\label{DECMARTSI1}
\end{align}
where $\xi_{n+1}(i)$ reduces to $\xi_{n+1}(i) = X_{n+1}^2(i) - \dE[X_{n+1}^2(i)| \cF_n]$. Hence, it follows from \eqref{DEFAN} and \eqref{DECMARTSI1} that for all $1 \leq i \leq d$ and for all $n\geq 2$,
\begin{equation}
\label{DECMARTSI2}
\sigma_n^2(i)= \frac{1}{a_n} \left(X_1^2(i)+2q \sum_{k=1}^{n-1} a_{k+1} \frac{\sigma_k^2}{k}  +N_n(i) \right)
\end{equation}
where
$$
N_n(i)= \sum_{k=2}^n a_k \xi_k(i).
$$
Hereafter, it is necessary to study the asymptotic behavior of the sequence $(N_n(i))$. We clearly have
$N_{n+1}(i)=N_n(i)+a_{n+1} \xi_{n+1}(i)$. Consequently, 
$$\dE[ N_{n+1}(i) |\cF_n]=
N_n(i) + a_{n+1} \dE[\xi_{n+1}(i) | \cF_n]=N_n(i) \hspace{1cm} \text{a.s.}
$$
which means that
$(N_n(i))$ is a square-integrable real martingale. Its predictable quadratic variation $\langle N(i) \rangle_n$
is given by
$$
\langle N(i) \rangle_n= \sum_{k=1}^{n-1}  a_{k+1}^2 \dE[\xi_{k+1}^2(i) | \cF_k].
$$ 
However, for all $1 \leq i \leq d$, the conditional distribution of $X_{n+1}^2(i)$ given $\cF_n$ is the Bernoulli
$\mathcal{B}(p_n(i))$ distribution where
$$
p_n(i)= \dE[X_{n+1}^2(i) | \cF_n]= \frac{a}{n}\sigma_n^2(i)+\frac{2q}{n}\sigma_{n}^2\leq 
\Big(\frac{p+q}{n}\Big)\sigma_{n}^2.
$$
It implies that $\dE[\xi_{n+1}^2(i) | \cF_n]=p_n(i)(1-p_n(i))$ a.s. Therefore, we obtain that for all $1 \leq i \leq d$, $\langle N(i) \rangle_n \leq (p+q) w_{n}$ a.s. 
where
\begin{equation}
\label{DEFWN}
w_n=\sum_{k=1}^{n} \frac{a_{k}^2 \sigma_k^2}{k}.
\end{equation}
Furthermore, let $(v_n)$ be the sequence defined, for all $n \geq 1$, by
\begin{equation}
\label{DEFVN}
v_n=\sum_{k=1}^{n} \frac{a_{k}^2}{kb_k}.
\end{equation}
We already saw in Section \ref{S-I} that the MERWS shows three regimes depending on the location of the memory parameter $p$ with respect to the
critical value $p_{d,r}$ given by \eqref{DEFPDR}. Since $b=1-r$, one can easily see that
$$
p<p_{d,r} \Longleftrightarrow 2a<b, \hspace{1cm}p=p_{d,r} \Longleftrightarrow 2a=b, \hspace{1cm}p>p_{d,r} \Longleftrightarrow 2a>b.
$$
Moreover, by taking the trace on both sides of \eqref{CALCIPM}, one can observe from \eqref{DEFB} and
\eqref{DEFWN} that we always have
$\text{Tr}(\langle M \rangle_n) \leq 1 +bw_n$.
In the diffusive regime where $2a<b$, we obtain from \eqref{LIMAN} and \eqref{LIMBN} that
\begin{equation}
\label{LIMVNDR}
\lim_{n \rightarrow \infty} \frac{v_n}{n^{b-2a}}=\ell_D
\hspace{1cm}\text{where}\hspace{1cm}
\ell_D=\frac{1}{(b-2a)}\frac{\Gamma^2(a+1)}{\Gamma(b+1)}.
\end{equation}
Then, we deduce from \eqref{LIMSIGMAN2} together with \eqref{DEFWN}, \eqref{DEFVN} and Toeplitz's lemma that
\begin{equation*}
\lim_{n \rightarrow \infty} \frac{w_n}{v_n}=\Gamma(b+1) \Sigma  \hspace{1cm} \text{a.s.}
\end{equation*}
which leads via \eqref{LIMVNDR} to
\begin{equation}
\label{CVGWNDR}
\lim_{n \rightarrow \infty} \frac{w_n}{n^{b-2a}}
= \frac{\Gamma^2(a+1)}{(b-2a)}\Sigma  \hspace{1cm} \text{a.s.}
\end{equation}
Consequently, it follows from \eqref{CVGWNDR} and the standard strong law of large numbers for martingales given e.g. by Theorem 1.3.24 in \cite{Duflo1997} that for all $1 \leq i \leq d$,
\begin{equation}
\label{ASCVGNID}
(N_n(i))^2=O(w_n \log w_n)=O(n^{b-2a} \log n) \hspace{1cm} \text{a.s.}
\end{equation}
Moreover, in the critical regime where $2a=b$, we find from \eqref{LIMAN} and \eqref{LIMBN} that
\begin{equation}
\label{LIMVNCR}
\lim_{n \rightarrow \infty} \frac{v_n}{\log n}=\ell_C
\hspace{1cm}\text{where}\hspace{1cm}
\ell_C=\frac{\Gamma^2(a+1)}{\Gamma(2a+1)}.
\end{equation}
Hence, as before, Toeplitz's lemma ensures that
\begin{equation}
\label{CVGWNCR}
\lim_{n \rightarrow \infty} \frac{w_n}{\log n}
= \Gamma^2(a+1)\Sigma  \hspace{1cm} \text{a.s.}
\end{equation}
Consequently, we obtain once again from the standard strong law of large numbers for martingales together with \eqref{CVGWNCR}
that for all $1 \leq i \leq d$,
\begin{equation}
\label{ASCVGNIC}
(N_n(i))^2=O(w_n \log w_n)=O( \log n \log \log n) \hspace{1cm} \text{a.s.}
\end{equation}
Finally, in the superdiffusive regime where $2a>b$, $(v_n)$ converges to the finite value 
\begin{equation} 
\label{LIMVNSR}
\lim_{n \rightarrow \infty}  v_n \!=\!
\sum_{k=0}^\infty \! \left(\frac{\Gamma(a+1)\Gamma(k+1)}{\Gamma(k+a+1)} \right)^{\!2}\!\!
\frac{\Gamma(k+b+1)}{\Gamma(k+2)\Gamma(b+1)}  \!=\! 
{}_{4}F_3 \!\left(\! \begin{matrix}
{\hspace{0.1cm}1 , 1 , 1 ,b+1}\\
{2,a+1,a+1}\end{matrix} \Big|
{\displaystyle 1}\!\right)\!
\end{equation} 
where $\!{}_{4}F_3$ stands for the hypergeometric function defined, for all $z \in \dC$, by
\begin{equation}
{}_{4}F_3 \left( \begin{matrix}
{a,b,c,d}\\
{e,f,g}\end{matrix} \Bigl|
{\displaystyle z}\right)
=\sum_{k=0}^{\infty}
\frac{(a)^{(k)}\, (b)^{(k)}\, (c)^{(k)} \,(d)^{(k)}}
{(e)^{(k)}\,(f)^{(k)}\,(g)^{(k)}\, k!} z^k.
\notag
\end{equation}
However, we already saw in Lemma \ref{L-MARTN} that the sequence $(N_n)$ is a martingale such that
for all $n \geq 1$, $\dE[N_n]=\dE[b_n \sigma_n^2]=1$. It implies that for all $n\geq 1$, $\dE[w_n]=v_n$.
Therefore, we deduce from the monotone convergence theorem that
$$
\sup_{n \geq 1} \dE[w_n] < \infty.
$$
Hence, we obtain that for all $1 \leq i \leq d$, 
$$\sup_{n \geq 1}\dE[N_n^2(i)] = \sup_{n \geq 1}\dE[\langle N(i) \rangle_n] \leq (p+q)\sup_{n \geq 1} \dE[w_n]
< \infty.
$$
Consequently, $(N_n(i))$ is a martingale bounded in $\dL^2$ and it follows from Doob's
martingale convergence theorem given e.g. by Corollary 2.2 in \cite{Hall1980}
that there exists a square integrable random variable $N(i)$ such that
\begin{equation}
\label{ASCVGNISD}
\lim_{n\rightarrow \infty} N_n(i)=N(i) \hspace{1cm} \text{a.s.}
\end{equation}
In the three regimes, we find from \eqref{ASCVGNID},  \eqref{ASCVGNIC} and  \eqref{ASCVGNISD} that for all $1 \leq i \leq d$, 
\begin{equation}
\label{LIMASCVGNI}
\lim_{n\rightarrow \infty} \frac{1 }{n^{b-a}}N_n(i)=0 \hspace{1cm} \text{a.s.}
\end{equation}
Moreover, we obtain from \eqref{LIMSIGMAN2} and Toeplitz's lemma that 
\begin{equation}
\label{TOPNI}
\lim_{n\rightarrow \infty} \frac{1 }{n^{b-a}} \sum_{k=1}^{n}  \frac{a_{k} \sigma_k^2}{k} = 
\frac{\Gamma(a+1)}{(b-a)}\Sigma 
\hspace{1cm} \text{a.s.}
\end{equation}
Therefore, as $b-a=2dq$, it follows from \eqref{DECMARTSI2} together with the almost sure convergences \eqref{LIMASCVGNI} and \eqref{TOPNI} that
for all $1 \leq i \leq d$,
\begin{equation}
\label{LIMSIGMANI}
\lim_{n\rightarrow \infty} \frac{1 }{n^{b}} \sigma_n^2(i) = \frac{1}{d}\Sigma 
\hspace{1cm} \text{a.s.}
\end{equation} 
Hence, we immediately obtain from \eqref{LIMSIGMANI} and the diagonal decomposition \eqref{GRAMDIAG} of the Gram matrix $\Sigma_n$ that
\begin{equation}
\label{LIMGRAM}
\lim_{n\rightarrow \infty} \frac{1 }{n^{b}} \Sigma_n = \frac{1}{d}\Sigma I_d
\hspace{1cm} \text{a.s.}
\end{equation} 
which completes the proof of Lemma \ref{L-ML}.
\demend


\vspace{-1ex}
\section*{Appendix A \\ Proofs in the diffusive regime}
\renewcommand{\thesection}{\Alph{section}}
\renewcommand{\theequation}{\thesection.\arabic{equation}}
\setcounter{section}{1}
\setcounter{equation}{0}
\label{S-A}


\subsection{Almost sure convergence.}
We start with the proof of the almost sure convergence in the diffusive regime where $2a<b$.

\ \vspace{-1ex}\\
\noindent{\bfseries Proof of Theorem \ref{T-ASCVG-DR}.}
It follows from \eqref{CALCIPM} and \eqref{DEFWN} that
$$
\text{Tr}(\langle M \rangle_n) =O(w_n) \hspace{1cm} \text{a.s.} 
$$
Then, the almost sure convergence \eqref{CVGWNDR} implies that $\text{Tr}(\langle M \rangle_n)=O(n^{b-2a})$ a.s.
Consequently, we deduce from the strong law of large numbers for multidimensional martingales given by the last part of Theorem 4.3.16 in \cite{Duflo1997} that 
\begin{equation}
\label{MNDR}
\| M_n\|^2=O(n^{b-2a} \log n) \hspace{1cm} \text{a.s.} 
\end{equation}
Therefore, as $M_n = a_nS_n$, we have from \eqref{LIMAN} and \eqref{MNDR} that
$\|S_n\|^2=O(n^{b} \log n)$ a.s. Finally, as $b=1-r<1$, we clearly obtain \eqref{ASCVGDR}
and \eqref{ASCVGDRSHARP}, which completes the proof of Theorem \ref{T-ASCVG-DR}.
\demend

\vspace{-2ex}
\subsection{Law of iterated logarithm.}
We now proceed to the proof of the law of iterated logarithm in the diffusive regime where $2a<b$.

\ \vspace{-1ex}\\
\noindent{\bfseries Proof of Theorem \ref{T-LIL-DR}.}
On the one hand, it follows from \eqref{LIMAN} and \eqref{LIMBN} together with
convergence \eqref{LIMGRAM} and Toeplitz's lemma that
\begin{equation}
\label{PLILDR1}
\lim_{n \rightarrow \infty} \frac{1}{n^{b-2a}}
\sum_{k=1}^{n} \left(\frac{a_{k}^2}{k} \right)\Sigma_k
= \frac{\Gamma^2(a+1)}{d(b-2a)}\Sigma I_d  \hspace{1cm} \text{a.s.} 
\end{equation}
By taking the trace on both sides of \eqref{PLILDR1}, we also get that
\begin{equation}
\label{PLILDR2}
\lim_{n \rightarrow \infty} \frac{1}{n^{b-2a}}
\sum_{k=1}^{n} \frac{a_{k}^2 \sigma_k^2}{k} 
= \frac{\Gamma^2(a+1)}{(b-2a)}\Sigma  \hspace{1cm} \text{a.s.}  
\end{equation}
On the other hand, we obtain from \eqref{ASCVGDRSHARP} that
\begin{equation*}
\lim_{n \rightarrow \infty} \frac{\|S_n\|^2}{n^{1+b}}
=0 \hspace{1cm} \text{a.s.}  
\end{equation*}
Hence, we get from \eqref{LIMAN}, \eqref{LIMBN} and Toeplitz's lemma that
\begin{equation}
\label{PLILDR3}
\lim_{n \rightarrow \infty} \frac{1}{n^{b-2a}}
\sum_{k=1}^{n} \left(\frac{a_{k}^2}{k} \right) S_kS_k^T
= 0  \hspace{1cm} \text{a.s.} 
\end{equation}
Consequently, we deduce from the conjunction of \eqref{CALCIPM} and \eqref{CVGWNDR} together with \eqref{PLILDR1}, \eqref{PLILDR2} and \eqref{PLILDR3} that
\begin{equation}
\label{CVGIPMNDR}
\lim_{n \rightarrow \infty} \frac{1}{w_n}\langle M \rangle_n= \frac{b}{d}I_d
\hspace{1cm}\text{a.s.}
\end{equation}
Hereafter, we already saw from \eqref{CEMERWS} and \eqref{CEPOSMERWS} that
\begin{equation}
\label{CMA1}
\dE[X_{n+1}| \cF_n] = \frac{a}{n}S_n  \hspace{1cm} \text{and}  \hspace{1cm}
\dE[S_{n+1}| \cF_n] = \left(1+\frac{a}{n}\right)S_n  \hspace{1cm} \text{a.s.}
\end{equation}
In addition, by taking the trace on both sides of \eqref{CVARMERWS} and \eqref{CVARSN}, we obtain that 
\begin{align}
\label{CNORMA2}
\dE[\|X_{n+1}\|^2 | \cF_n] &= \frac{b\sigma_n^2}{n} \hspace{1cm} \text{a.s.}  \\
\dE[\|S_{n+1}\|^2 | \cF_n] &= \left(1+\frac{2a}{n}\right)\|S_n\|^2 + \frac{b\sigma_n^2}{n}  \hspace{1cm} \text{a.s.}
\label{CNORMA3}
\end{align}
Moreover, it is easy to see from \eqref{STEPSMERWS} that $\dE[X_{n+1}\|X_{n+1}\|^2 | \cF_n] = \dE[X_{n+1}| \cF_n]$ and 
$\dE[\|X_{n+1}\|^4 | \cF_n] = \dE[\|X_{n+1}\|^2 | \cF_n]$ almost surely. 
Therefore, it follows from \eqref{CVARMERWS}, \eqref{CVARSN}, \eqref{CMA1}, \eqref{CNORMA2} and \eqref{CNORMA3} together with tedious but straightforward calculations that
\begin{eqnarray}
\dE[\|S_{n+1}\|^2 S_{n+1} | \cF_n] &=& \left(1+\frac{3a}{n}\right)\|S_n\|^2 S_n +\frac{a}{n}S_n  - \frac{2a \sigma_n^2}{n}S_n\nonumber \\
& + & \frac{2a}{n} \Sigma_n S_n   + \frac{3b\sigma_n^2}{n}S_n   
\hspace{1cm} \text{a.s.} 
\label{CM3A3}
\end{eqnarray}
and
\begin{eqnarray}
\dE[\|S_{n+1}\|^4 | \cF_n] &=& \left(1+\frac{4a}{n}\right)\|S_n\|^4 + \frac{4a}{n} \Big(S_n^T\Sigma_n S_n + \|S_n\|^2- \sigma_n^2 \|S_n\|^2\Big)
 \nonumber \\
& + &  \frac{b \sigma_n^2}{n} (1+6 \|S_n\|^2)
\hspace{1cm} \text{a.s.} 
\label{CM4A4}
\end{eqnarray}
Hence, as $\varepsilon_{n+1}=S_{n+1} - \alpha_n S_n$, \eqref{CM3A3} and \eqref{CM4A4} lead to
\begin{eqnarray}
\dE[\|\varepsilon_{n+1}\|^2 \varepsilon_{n+1} | \cF_n] &=& 2\Big(\frac{a}{n}\Big)^3 \|S_n\|^2 S_n +\frac{a}{n}S_n  + 2\Big(\frac{a}{n}\Big)^2\sigma_n^2 S_n \nonumber \\
& - & 2\Big(\frac{a}{n}\Big)^2 \Sigma_n S_n -  \frac{ 3ab \sigma_n^2}{n^2}S_n 
\hspace{1cm} \text{a.s.} 
\label{CM3EPS}
\end{eqnarray}
and
\begin{eqnarray}
\dE[\|\varepsilon_{n+1}\|^4|\cF_n] &=& \frac{b \sigma_n^2}{n} \left(1+ 6\left(\frac{a\|S_n\|}{n}\right)^2 \right) -3\left(\frac{a\|S_n\|}{n}\right)^4 -4\left(\frac{a\|S_n\|}{n}\right)^2 
 \nonumber \\
& + &  -4\Big(\frac{a}{n}\Big)^3 \Big( \sigma_n^2 \|S_n\|^2-    S_n^T\Sigma_n S_n\Big) 
\hspace{1cm} \text{a.s.} 
\label{CM4EPS}
\end{eqnarray}
Consequently, we deduce from \eqref{CM4EPS} that for all $n\geq 1$,
\begin{equation}
\label{CMOMEPS4}
\dE[\|\varepsilon_{n+1}\|^4|\cF_n] \leq \frac{7b \sigma_n^2}{n} \hspace{1cm} \text{a.s.}
\end{equation}
By taking the expectation on both sides of \eqref{CMOMEPS4}, we obtain that for all $n\geq 1$, 
\begin{equation}
\label{MOM4EPS}
 \dE[\|\varepsilon_{n+1}\|^4]\leq \frac{7b}{n b_n} \leq \frac{7n^b}{\Gamma(b) n}
\end{equation}
where the last upper bound is due to \eqref{DEFBN} and to Wendel's inequality for the ratio of gamma functions.
\begin{equation}
\label{UBMOM4EPS}
 \dE[||\varepsilon_{n+1}||^4 ]\leq \frac{7}{\Gamma(b)n^{1-b}}.
\end{equation}
We are now in position to prove Theorem \ref{T-LIL-DR}. For any vector $u \in \dR^d$, denote $M_n(u)= \langle u, M_n \rangle$, 
$\varepsilon_n(u)= \langle u, \varepsilon_n \rangle$ and $\Delta M_n(u)= a_n \varepsilon_n(u)$. 
We immediately find from the almost sure convergence \eqref{CVGIPMNDR} that 
\begin{equation}
\label{CVGIPMNDRU}
\lim_{n \rightarrow \infty} \frac{1}{w_n}\langle M(u) \rangle_n=  \frac{b}{d} \|u\|^2
\hspace{1cm}\text{a.s.}
\end{equation}
Moreover, it follows from \eqref{UBMOM4EPS} and the Cauchy-Schwarz inequality that
\begin{equation}
\label{BSUMDR}
\sum_{n=2}^\infty \frac{1}{n^{2(b-2a)}} \dE[|\Delta M_n(u)|^4] 
\leq 
\frac{7 \|u\|^4}{\Gamma(b)}\sum_{n=1}^\infty \frac{a_{n}^4}{n^{b+1-4a}} <\infty.
\end{equation}
Furthermore, let $(P_n(u))$ be the martingale defined, for all $n \geq 1$, by
$$
P_n(u)=\sum_{k=1}^n \frac{a_k^2}{k^{b-2a}} \Big(\varepsilon_k^2(u) - \dE[\varepsilon_k^2(u) | \cF_{k-1}]\Big).
$$
Its predictable quadratic variation is given by
$$
\langle P(u) \rangle_n = \sum_{k=1}^n \frac{a_k^4}{k^{2(b-2a)}}
\Big(\dE[\varepsilon_k^4(u)| \mathcal{F}_{k-1}]- \dE^2[\varepsilon_k^2(u) | \cF_{k-1}]\Big).
$$
Hence, we obtain from \eqref{BSUMDR} that $(P_n(u))$ is bounded in $\dL^2$ as
\begin{equation*}
\sup_{n \geq 1} \dE\big[\langle P(u) \rangle_n \big] < \infty.
\end{equation*}
Consequently, we deduce from Doob's martingale convergence theorem \cite{Hall1980}
that $(P_n(u))$ converges a.s.
Then, it follows from the law of iterated logarithm given by Theorem 1 and Corollary 2 in \cite{Heyde1977} that for any vector $u \in \dR^d$, 
\begin{eqnarray}
 \limsup_{n \rightarrow \infty} 
\left(\frac{1}{2w_n \log \log w_n}\right)^{1/2} M_n(u) 
  & = &
 -\liminf_{n \rightarrow \infty} 
 \left(\frac{1}{2w_n \log \log w_n}\right)^{1/2} M_n(u)  \nonumber \\
& = & \left(\frac{b}{d} \right)^{\!1/2}\| u \|\hspace{1cm} \text{a.s.}
\label{PLILDR4}
\end{eqnarray}
Therefore, as $M_n(u)=a_n \langle u, S_n\rangle$, we obtain from \eqref{LIMAN}, \eqref{LIMSIGMAN2}, \eqref{CVGWNDR}  and \eqref{PLILDR4} that
\begin{eqnarray}
 \limsup_{n \rightarrow \infty} 
\left(\frac{1}{2 \sigma_n^2 \log \log \sigma_n^2}\right)^{1/2} \langle u, S_n\rangle 
  & = &
 -\liminf_{n \rightarrow \infty} 
 \left(\frac{1}{2 \sigma_n^2 \log \log \sigma_n^2}\right)^{1/2} \langle u, S_n\rangle   \nonumber \\
& = & \left(\frac{b}{d(b-2a)} \right)^{\!1/2}\| u \|\hspace{1cm} \text{a.s.}
\label{PLILDR5}
\end{eqnarray}
In particular, we get from \eqref{PLILDR5} that for any vector $u$ of $\dR^d$,
\begin{equation}
\label{PLIL-SNU-DR}
 \limsup_{n \rightarrow \infty} \frac{\langle u, S_n\rangle^2}{2 \sigma_n^2 \log \log \sigma_n^2}
  = \frac{b}{d(b-2a)}\|u\|^2 \hspace{1cm} \text{a.s.}
\end{equation}
However, we have the decomposition
$$
\| S_n \|^2=\sum_{i=1}^d \langle e_i, S_n\rangle^2
$$
where $(e_1, \ldots, e_d)$ is the standard basis of $\dR^d$. Finally, we deduce from
\eqref{PLIL-SNU-DR} that
\begin{equation}
\label{PLIL-SN-DR}
 \limsup_{n \rightarrow \infty} \frac{ \|S_n\|^2}{2 \sigma_n^2 \log \log \sigma_n^2}
  = \frac{b}{(b-2a)} \hspace{1cm} \text{a.s.}
\end{equation}
which clearly leads to \eqref{LIL-DR1} as it comes from \eqref{VARDR} and the elementary fact that $a=p-q$ and $b=1-r$ that
\begin{equation}
\label{PVAR-DR}
\frac{b}{(b-2a)}=\frac{(2d-1)b}{(2d+1)b-4dp}=dv^2.
\end{equation}
We find \eqref{LIL-DR2} from \eqref{LIMSIGMAN2} 
and \eqref{PLIL-SN-DR} which completes the proof of Theorem \ref{T-LIL-DR}. \demend

\vspace{-2ex}
\subsection{Asymptotic normality.}
We continue with the proof of the asymptotic normality in the diffusive regime where $2a<b$
using the Cram\'er-Wold theorem.

\ \vspace{-1ex}\\
\noindent{\bfseries Proof of Theorem \ref{T-AN-DR}.} 
We already saw that the almost sure convergence \eqref{CVGIPMNDRU} holds for any vector $u \in \dR^d$
and that $(P_n(u))$ converges a.s. In order to make use of
Theorem 1 and Corollaries 1 and 2 in \cite{Heyde1977}, it only remains to show that for any vector $u \in \dR^d$ and
that for any $\eta>0$,
\begin{equation}
\lim_{n \rightarrow \infty}
\frac{1}{n^{b-2a}}\sum_{k=1}^{n}\dE\Big[\Delta M_k^2(u) \rI_{\big\{|\Delta M_k(u)|>\eta \sqrt{n^{b-2a}} \big \}}\Big]=0.
\label{PANDR1}
\end{equation}
We have for any $\eta>0$,
\begin{equation*}
\frac{1}{n^{b-2a}}\sum_{k=1}^{n}\dE\Big[\Delta M_k^2(u) \rI_{\big\{|\Delta M_k(u)|>\eta \sqrt{n^{b-2a}} \big \}}\Big]
\leq 
\frac{1}{\eta^2 n^{2(b-2a)}}\sum_{k=1}^{n}\dE\Big[\Delta M_k^4(u)\Big].
\end{equation*}
However, it follows from Kronecker's lemma and \eqref{BSUMDR} that for any vector $u \in \dR^d$,
\begin{equation*}
\lim_{n \rightarrow \infty}
\frac{1}{n^{2(b-2a)}}\sum_{k=1}^{n}\dE\Big[\Delta M_k^4(u)\Big]=0
\end{equation*}
which clearly implies \eqref{PANDR1}. Hence, all the conditions of Theorem 1 and Corollaries 1 and 2 in \cite{Heyde1977}
are satisfied, which leads to the asymptotic normality
\begin{equation}
\label{PANDR2}
\frac{M_n(u)}{\sqrt{w_n}} \underset{n\rightarrow+\infty}{\overset{\cL}{\longrightarrow}} \cN\Big(0, \frac{b}{d}\|u\|^2\Big).
\end{equation}
Therefore, as $M_n(u)=a_n \langle u, S_n\rangle$,  we obtain from \eqref{LIMAN}, \eqref{LIMSIGMAN2}, \eqref{CVGWNDR}  and \eqref{PANDR2} that
\begin{equation}
\label{PANDR3}
\frac{\langle u, S_n\rangle}{\sqrt{\sigma^2_n}} \underset{n\rightarrow+\infty}{\overset{\cL}{\longrightarrow}} \cN\Big(0, \frac{b}{d(b-2a)}\|u\|^2\Big).
\end{equation}
Consequently, we deduce the asymptotic normality \eqref{ANDR} from \eqref{PVAR-DR} and \eqref{PANDR3} together with the Cram\'er-Wold theorem.
Moreover, we also find from Theorem 1 in \cite{Heyde1977} that for any vector $u \in \dR^d$,
\begin{equation}
\label{PANDR4}
\frac{M_n(u)}{\sqrt{n^{b-2a}}} \underset{n\rightarrow+\infty}{\overset{\cL}{\longrightarrow}} \Gamma(a+1) \sqrt{\Sigma^\prime}\cN\Big(0, \frac{b}{d(b-2a)}\|u\|^2\Big)
\end{equation}
where $\Sigma^\prime$ is independent of the Gaussian distribution at the right-hand side of \eqref{PANDR4}
and $\Sigma^\prime$ has a $\cM\cL(b)$ distribution. Finally, it follows from \eqref{LIMAN} and \eqref{PANDR4} that
\begin{equation*}
\frac{\langle u, S_n\rangle}{\sqrt{n^b}} \underset{n\rightarrow+\infty}{\overset{\cL}{\longrightarrow}} \sqrt{\Sigma^\prime}\cN\Big(0, \frac{b}{d(b-2a)}\|u\|^2\Big)
\end{equation*}
which implies \eqref{ANDRML} and achieves the proof of Theorem \ref{T-AN-DR}.
\demend


\vspace{-2ex}
\section*{Appendix B \\ Proofs in the critical regime}
\renewcommand{\thesection}{\Alph{section}}
\renewcommand{\theequation}{\thesection.\arabic{equation}}
\setcounter{section}{2}
\setcounter{equation}{0}
\setcounter{subsection}{0}
\label{S-B}


\subsection{Almost sure convergence.}
We carry on with the proof of the almost sure convergence in the critical regime where $2a=b$. 

\ \vspace{-1ex}\\
\noindent{\bfseries Proof of Theorem \ref{T-ASCVG-CR}.}
We already saw from \eqref{CALCIPM} and \eqref{DEFWN} that
$$
\text{Tr}(\langle M \rangle_n) =O(w_n) \hspace{1cm} \text{a.s.} 
$$
The almost sure convergence \eqref{CVGWNCR} implies that $\text{Tr}(\langle M \rangle_n)=O(\log n)$ a.s.
Therefore, we obtain once again from the strong law of large numbers for multidimensional martingales given by the last part of Theorem 4.3.16 in \cite{Duflo1997} that 
\begin{equation}
\label{MNCR}
\| M_n\|^2=O(\log n \log\log n) \hspace{1cm} \text{a.s.} 
\end{equation}
Consequently, as $M_n = a_nS_n$ and due to the fact that $2a=b$, we find from \eqref{LIMAN} 
and \eqref{MNCR} that $\|S_n\|^2=O(n^{b} \log n \log \log n)$ a.s.  Hence, as $b=1-r<1$, it clearly leads to \eqref{ASCVGCR}
and \eqref{ASCVGCRSHARP}, which completes the proof of Theorem \ref{T-ASCVG-CR}.
\demend

\vspace{-2ex}
\subsection{Law of iterated logarithm.} 
We now focus on the proof of the law of iterated logarithm in the critical regime where $2a=b$. 

\ \vspace{-1ex}\\
\noindent{\bfseries Proof of Theorem \ref{T-LIL-CR}.}
The proof is quite similar to that of Theorem \ref{T-LIL-DR}. We have from \eqref{LIMAN}, \eqref{LIMBN}, \eqref{LIMGRAM} and Toeplitz's lemma that
\begin{equation}
\label{PLILCR1}
\lim_{n \rightarrow \infty} \frac{1}{\log n}
\sum_{k=1}^{n} \left(\frac{a_{k}^2}{k} \right)\Sigma_k
= \frac{\Gamma^2(a+1)}{d}\Sigma I_d  \hspace{1cm} \text{a.s.} 
\end{equation}
By taking the trace on both sides of \eqref{PLILCR1}, we also get that
\begin{equation}
\label{PLILCR2}
\lim_{n \rightarrow \infty} \frac{1}{\log n}
\sum_{k=1}^{n} \frac{a_{k}^2 \sigma_k^2}{k} 
= \Gamma^2(a+1)\Sigma  \hspace{1cm} \text{a.s.}  
\end{equation}
Moreover, it also follows \eqref{ASCVGCRSHARP}, \eqref{LIMAN}, \eqref{LIMBN} and Toeplitz's lemma that
\begin{equation}
\label{PLILCR3}
\lim_{n \rightarrow \infty} \frac{1}{\log n}
\sum_{k=1}^{n} \left(\frac{a_{k}^2}{k} \right) S_kS_k^T
= 0  \hspace{1cm} \text{a.s.} 
\end{equation}
Hence, we obtain from \eqref{CALCIPM} and \eqref{CVGWNCR} together with \eqref{PLILCR1}, \eqref{PLILCR2} and \eqref{PLILCR3} that
\begin{equation}
\label{CVGIPMNCR}
\lim_{n \rightarrow \infty} \frac{1}{w_n}\langle M \rangle_n= \frac{b}{d}I_d
\hspace{1cm}\text{a.s.}
\end{equation}
For any vector $u \in \dR^d$, let $M_n(u)= \langle u, M_n \rangle$, 
$\varepsilon_n(u)= \langle u, \varepsilon_n \rangle$ and $\Delta M_n(u)= a_n \varepsilon_n(u)$. 
The almost sure convergence \eqref{CVGIPMNCR} ensures that 
\begin{equation}
\label{CVGIPMNCRU}
\lim_{n \rightarrow \infty} \frac{1}{w_n}\langle M(u) \rangle_n=  \frac{b}{d} \|u\|^2
\hspace{1cm}\text{a.s.}
\end{equation}
In addition, we have from \eqref{UBMOM4EPS} and the Cauchy-Schwarz inequality that
\begin{equation}
\label{BSUMCR}
\sum_{n=2}^\infty \frac{1}{(\log n)^2} \dE[|\Delta M_n(u)|^4] 
\leq 
\frac{7 \|u\|^4}{\Gamma(b)}\sum_{n=2}^\infty  \frac{1}{(\log n)^2} \frac{a_{n}^4}{n^{1-b}} <\infty.
\end{equation}
Furthermore, let $(Q_n(u))$ be the martingale defined, for all $n \geq 1$, by
$$
Q_n(u)=\sum_{k=2}^n \frac{a_k^2}{\log k} \Big(\varepsilon_k^2(u) - \dE[\varepsilon_k^2(u) | \cF_{k-1}]\Big).
$$
As in the proof of Theorem \ref{T-LIL-DR}, we obtain that $(Q_n(u))$ converges a.s. Consequently, we deduce from the law of iterated logarithm given by
Theorem 1 and Corollary 2 in \cite{Heyde1977} that for any vector $u \in \dR^d$,
\begin{eqnarray*}
 \limsup_{n \rightarrow \infty} 
\left(\frac{1}{2w_n \log \log w_n}\right)^{1/2} M_n(u) 
  & = &
 -\liminf_{n \rightarrow \infty} 
 \left(\frac{1}{2w_n \log \log w_n}\right)^{1/2} M_n(u)  \nonumber \\
& = & \left(\frac{b}{d} \right)^{\!1/2}\| u \|\hspace{1cm} \text{a.s.}
\end{eqnarray*}
which leads via \eqref{LIMAN} and \eqref{LIMSIGMAN2} to \eqref{LIL-CR1}.
Finally, \eqref{LIL-CR2} follows from \eqref{LIL-CR1} and \eqref{LIMSIGMAN2} which achieves the proof of
Theorem \ref{T-LIL-CR}. \demend

\vspace{-2ex}
\subsection{Asymptotic normality.}
We proceed to the proof of the asymptotic normality in the critical regime where $2a=b$.

\ \vspace{-1ex}\\
\noindent{\bfseries Proof of Theorem \ref{T-AN-CR}.}
Via the same arguments as in the proof of Theorem \ref{T-AN-DR}, we obtain that for any vector $u \in \dR^d$, 
\begin{equation}
\label{PANCR1}
\frac{M_n(u)}{\sqrt{w_n}} \underset{n\rightarrow+\infty}{\overset{\cL}{\longrightarrow}} \cN\Big(0, \frac{b}{d}\|u\|^2\Big).
\end{equation}
Hence, as $M_n(u)=a_n \langle u, S_n\rangle$,  we find from \eqref{LIMAN}, \eqref{LIMSIGMAN2}, \eqref{CVGWNCR}  and \eqref{PANCR1} that
\begin{equation}
\label{PANCR2}
\frac{\langle u, S_n\rangle}{\sqrt{\sigma^2_n \log \sigma^2_n}} \underset{n\rightarrow+\infty}{\overset{\cL}{\longrightarrow}} \cN\Big(0, \frac{1}{d}\|u\|^2\Big).
\end{equation}
Consequently, we deduce the asymptotic normality \eqref{ANCR} from \eqref{PANCR2} together with the Cram\'er-Wold theorem. 
The proof of \eqref{ANCRML} is left to the reader inasmuch as it follows the same lines as that of \eqref{ANDRML}.
\demend


\vspace{-2ex}
\section*{Appendix C \\ Proofs in the superdiffusive regime}
\renewcommand{\thesection}{\Alph{section}}
\renewcommand{\theequation}{\thesection.\arabic{equation}}
\setcounter{section}{3}
\setcounter{equation}{0}
\setcounter{subsection}{0}
\label{S-C}

\subsection{Almost sure convergence.}
We carry on with the proof of the almost sure convergence in the superdiffusive regime where $2a>b$. 

\ \vspace{-1ex}\\
\noindent{\bfseries Proof of Theorem \ref{T-ASCVG-SR}.}
\noindent
We have from \eqref{CALCIPM} and \eqref{DEFWN} together with \eqref{LIMVNSR} that
$$
\lim_{n \rightarrow \infty} \text{Tr}(\langle M \rangle_n) < \infty \hspace{1cm} \text{a.s.} 
$$
Therefore, it follows from the second part of Theorem 4.3.15 in \cite{Duflo1997} and \eqref{NEWDEFMN} that 
\begin{equation}
\label{MNSR}
\lim_{n \rightarrow \infty} M_n=M \hspace{1cm} \text{a.s.} 
\end{equation}
where $M$ is the random vector of $\dR^d$ given by
$$
M=\sum_{k=1}^\infty a_k \veps_k.
$$
Consequently, as $M_n=a_n S_n$, we obtain from \eqref{LIMAN} and \eqref{MNSR} that
\begin{equation*}
\lim_{n \rightarrow \infty} \frac{S_n}{n^{a}}
= L  \hspace{1cm} \text{a.s.}
\end{equation*}
where the limiting random vector $L$ of $\dR^d$ given by
\begin{equation}
\label{DEFL}
L=\frac{1}{\Gamma(a+1)}\sum_{k=1}^\infty a_k \veps_k.
\end{equation}
Moreover, as $M_0=0$, we have from \eqref{NEWDEFMN}, \eqref{CALCIPM} and \eqref{DEFWN} that for all $n \geq 1$,
$$
\dE[\|M_n\|^2]=\dE[\text{Tr} (\langle M \rangle_n)] \leq 
1+b\dE[w_n] \leq 1+bv_n
$$
since $\dE[w_n]=v_n$, which leads via \eqref{LIMVNSR} to
\begin{equation}
\label{BOUNDIPMSR}
\sup_{n \geq 1} \dE\left[\|M_n\|^2\right] < \infty.
\end{equation} 
Therefore, we deduce from \eqref{BOUNDIPMSR} that the martingale $(M_n)$ is bounded in $\dL^2$.
Hence, it follows from Doob's martingale convergence theorem given e.g. by Corollary 2.2 in \cite{Hall1980} that
\begin{equation*}
\lim_{n \to \infty} \dE\bigl[\|M_n-M\|^2\bigr]=0,
\end{equation*}
which immediately implies the mean square convergence \eqref{ASCVGSR2}. In order to complete the proof of Theorem
\ref{T-ASCVG-SR}, it only remains to compute the mean and the covariance matrix associated with $L$.
We clearly have for all $n \geq 1$, $\dE[M_n]=0$, which ensures that 
$\dE[M]=0$ and $\dE[L]=0$. Moreover, by taking the expectation on both sides of \eqref{CVARSN}, we obtain
that for all $n\geq 1$, 
\begin{equation}
\label{CALCSNSNT}
   \dE[S_{n+1}S_{n+1}^T] =\left(1+\frac{2a}{n}\right) \dE[S_nS_n^T] + \frac{a}{n}\dE[\Sigma_n] + \frac{2q}{n}\dE[\sigma_n^2] I_d.
\end{equation}
On the one hand, we already saw in the proof of Lemma \ref{L-ML} that $(N_n)$ is a martingale such that
for all $n \geq 1$, $\dE[N_n]=\dE[b_n \sigma_n^2]=1$, which implies via \eqref{DEFBN} that 
\begin{equation}
\label{CALCMSIGMAN2}
\dE[\sigma_n^2]=\frac{1}{b_n}= \frac{(b)^{(n)}}{ b(n-1)!}.
\end{equation}
On the other hand, by taking the expectation on both sides of \eqref{DECMARTSI2}, we obtain from 
\eqref{CALCMSIGMAN2} that for all $1 \leq i \leq d$ and for all $n\geq 2$,
\begin{align}
\label{CALCMSIGMANI21}
\dE[\sigma_n^2(i)] &= \frac{1}{a_n} \left(\dE[X_1^2(i)]+2q \sum_{k=1}^{n-1} a_{k+1} \frac{\dE[\sigma_k^2]}{k} \right) \nonumber \\
&=\frac{1}{a_n} \left(\frac{1}{d}+\frac{2q}{b} \sum_{k=1}^{n-1} a_{k+1} \frac{(b)^{(k)}}{k!} \right).
\end{align}
Furthermore, it follows from \eqref{DEFAN} together with Lemma B.1 in \cite{Bercu2018} that
\begin{equation}
\label{CALCMSIGMANI22}
\sum_{k=1}^{n-1} a_{k+1} \frac{(b)^{(k)}}{k!} = \sum_{k=1}^{n-1}  \frac{a (b)^{(k)}}{(a)^{(k+1)}}
= \sum_{k=1}^{n-1}  \frac{(b)^{(k)}}{(a+1)^{(k)}} 
=\frac{b}{(b-a)}\left( \frac{a (b)^{(n)}}{b (a)^{(n)}} -1\right).
\end{equation}
Consequently, as $b-a=2dq$, we find from \eqref{CALCMSIGMANI21} and \eqref{CALCMSIGMANI22} that for all $1 \leq i \leq d$ and for all $n\geq 2$, $\dE[\sigma_n^2(i)]$ drastically reduces to
\begin{align}
\label{CALCMSIGMANI23}
\dE[\sigma_n^2(i)] 
&=\frac{1}{a_n} \left(\frac{1}{d}+\frac{2q b}{b(b-a)} \left( \frac{a (b)^{(n)}}{b (a)^{(n)}} -1\right) \right)
=\frac{1}{a_n} \left(\frac{1}{d}+\frac{1}{d} \left( \frac{a (b)^{(n)}}{b (a)^{(n)}} -1\right) \right) \nonumber \\
&=\frac{1}{d a_n} \left(1+ \frac{a (b)^{(n)}}{b (a)^{(n)}} -1\right)= \frac{(a)^{(n)}}{ a d (n-1)!} 
\left(\frac{a (b)^{(n)}}{b (a)^{(n)}} \right)=\frac{(b)^{(n)}}{ b d (n-1)!}= \frac{1}{d b_n}.
\end{align}
One can observe from \eqref{DECSIGMAN2} and \eqref{CALCMSIGMANI23} that 
\begin{equation*}
\dE[\sigma_n^2]=\sum_{i=1}^d \dE[\sigma_n^2(i)]= \sum_{i=1}^d \frac{1}{d b_n}= \frac{1}{b_n},
\end{equation*}
which is consistent with \eqref{CALCMSIGMAN2}. Hereafter, we immediately obtain from
\eqref{GRAMDIAG} and \eqref{CALCMSIGMANI23} that for all $n \geq 1$,
\begin{equation}
\label{CALCMGRAM}
\dE[\Sigma_n]=\frac{1}{db_n}I_d= \frac{(b)^{(n)}}{ bd (n-1)!}I_d.
\end{equation}
Hence, it follows from the conjunction of \eqref{CALCSNSNT}, \eqref{CALCMSIGMAN2} and \eqref{CALCMGRAM} that
for all $n \geq 1$,
\begin{equation*}
 \dE[S_{n+1}S_{n+1}^T] =\left(1+\frac{2a}{n}\right) \dE[S_nS_n^T] +\frac{(b)^{(n)}}{ d n!}I_d,
\end{equation*}
which implies
\begin{align}
\label{CALCSNSNTFIN}
 \dE[S_{n}S_{n}^T] &= \frac{(2a)^{(n)}}{ 2a (n-1)!} 
 \left( \dE[X_1X_1^T]+ \frac{1}{d}\sum_{k=1}^{n-1}  \frac{(b)^{(k)}}{(2a+1)^{(k)}} I_d \right) \nonumber \\
 &= \frac{(2  a)^{(n)}}{ 2  a  (n-1)!} \left( \frac{1}{d}I_d 
 + \frac{b}{d(2a-b)} \left( 1- \frac{2a (b)^{(n)}}{b (2a)^{(n)}}\right) I_d\right) \nonumber \\
&=\frac{(2a)^{(n)}}{ d(2a-b) (n-1)!} 
\left( 1- \frac{(b)^{(n)}}{(2a)^{(n)}}\right)I_d.
\end{align}
Finally, as $2a>b$, we deduce from \eqref{LIMAN} \eqref{DEFL} and \eqref{CALCSNSNTFIN} that
\begin{align*}
\dE[L L^T] &= \lim_{n \rightarrow \infty} \frac{1}{\Gamma^2(a+1)}\dE[M_n M_n^T]
= \lim_{n \rightarrow \infty} \frac{a_n^2}{\Gamma^2(a+1)}\dE[S_n S_n^T] \\
&= \lim_{n \rightarrow \infty} \frac{1}{n^{2a}}\dE[S_n S_n^T]
= \lim_{n \rightarrow \infty} \frac{(2a)^{(n)}}{d(2a-b) n^{2a} (n-1)!}I_d\\
&=
\frac{1}{d(2a-b) \Gamma(2a)}I_d
\end{align*}
which is exactly \eqref{COVL} as $a=p-q$ and $b=1-r$, which implies that
$$
\frac{1}{d(2a-b)}=\frac{2d-1}{d(4dp-(2d+1)b)d}.
$$
\demend

\vspace{-2ex}
\subsection{Gaussian fluctuations.}
We conclude these appendices with the proof of the Gaussian fluctuations in the superdiffusive regime where $2a>b$. 

\ \vspace{-1ex}\\
\noindent{\bfseries Proof of Theorem \ref{T-AN-SR}.}
For all $n \geq 1$, denote
$$
\vspace{-1ex}
\Lambda_n=\sum_{k=n}^{\infty} \dE\big[\Delta M_{k+1} \Delta M_{k+1}^T |\cF_{k}]= 
\sum_{k=n}^{\infty}a_{k+1}^2\dE\big[\varepsilon_{k+1} \varepsilon_{k+1}^T |\cF_{k}].
$$
It follows from \eqref{CVAREPSILON} that for all $n\geq 1$,
\begin{equation}
\label{PANSR1}
\Lambda_n= a\sum_{k=n}^{\infty}  \left(\frac{a_{k+1}^2}{k}\right) \Sigma_k
+2q \sum_{k=n}^{\infty}  \left(\frac{a_{k+1}^2}{k}\right) \sigma_k^2 I_d
-a^2 \sum_{k=n}^{\infty}  \left(\frac{a_{k+1}}{k}\right)^2 S_k S_k^T.
\end{equation}
On the one hand, we obtain from the almost sure convergence \eqref{LIMGRAM} that
\begin{equation}
\label{PANSR2}
\lim_{n \rightarrow \infty} n^{2a-b}\sum_{k=n}^{\infty}\left(\frac{a_{k+1}^2}{k}\right) \Sigma_k= \frac{\Gamma^2(a+1)}{d(2a-b)}  \Sigma I_d
\hspace{1cm}\text{a.s.}
\end{equation}
By taking the trace on both sides of \eqref{PANSR2}, we also find that
\begin{equation}
\label{PANSR3}
\lim_{n \rightarrow \infty} n^{2a-b}
\sum_{k=n}^{\infty} \left(\frac{a_{k+1}^2}{k}\right) \sigma_k^2 
= \frac{\Gamma^2(a+1)}{(2a-b)}\Sigma  \hspace{1cm} \text{a.s.}  
\end{equation}
On the other hand, we deduce from the almost sure convergence \eqref{ASCVGSR1} that
\begin{equation}
\label{PANSR4}
\lim_{n \rightarrow \infty} n\sum_{k=n}^{\infty}\left(\frac{a_{k+1}}{k}\right)^2 S_k S_k^T= \Gamma^2(a+1)  L L^T
\hspace{1cm}\text{a.s.}
\end{equation}
Consequently, it comes from \eqref{PANSR1}, \eqref{PANSR2}, \eqref{PANSR3} and \eqref{PANSR4} that
\begin{equation}
\label{PANSR5}
\lim_{n \rightarrow \infty} n^{2a-b}\Lambda_n= \Gamma^2(a+1)\vartheta^2  \Sigma I_d
\hspace{1cm}\text{a.s.}
\end{equation}
where the asymptotic variance $\vartheta^2$, given by \eqref{VARSR}, is such that
$$
\vartheta^2=\frac{b}{d(2a-b)}=\frac{(2d-1)b}{d(4dp -(2d+1))}.
$$ 
Denote for any vector $u \in \dR^d$, $M_n(u)= \langle u, M_n \rangle$, 
$\varepsilon_n(u)= \langle u, \varepsilon_n \rangle$, $\Delta M_n(u)= a_n \varepsilon_n(u)$
and $\Lambda_n(u)=u^T\Lambda_n u$. It follows from the almost sure convergence \eqref{PANSR5} that
\begin{equation}
\label{PANSR6}
\lim_{n \rightarrow \infty} n^{2a-b}\Lambda_n(u)= \Gamma^2(a+1)\vartheta^2  \Sigma \| u \|^2
\hspace{1cm}\text{a.s.}
\end{equation}
In addition, we claim that for any vector $u \in \dR^d$ and
that for any $\eta>0$,
\begin{equation}
\lim_{n \rightarrow \infty}
n^{2a-b}\sum_{k=n}^{\infty}\dE\Big[\Delta M_k^2(u) \rI_{\big\{|\Delta M_k(u)|>\eta \sqrt{n^{b-2a}} \big \}}\Big]=0.
\label{PANSR7}
\end{equation}
As a matter of fact, we have from \eqref{UBMOM4EPS} and the Cauchy-Schwarz inequality that for any $\eta>0$,
\begin{eqnarray}
n^{2a-b}\sum_{k=n}^{\infty} \dE\Big[\Delta M_k^2(u) \rI_{\big\{|\Delta M_k(u)|>\eta \sqrt{n^{b-2a}} \big \}}\Big]
&\leq &
\frac{1}{\eta^2}n^{2(2a-b)}\sum_{k=n}^{\infty}\dE\Big[\Delta M_k^4(u)\Big], \notag\\
&\leq &
\frac{7\|u\|^4}{\eta^2\Gamma(b)} n^{2(2a-b)}\sum_{k=n}^{\infty} \frac{a_k^4}{k^{1-b}}.
\label{PANSR8}
\end{eqnarray}
However, one can easily see that
\begin{equation}
\lim_{n \rightarrow \infty} n^{4a-b}\sum_{k=n}^{\infty} \frac{a_k^4}{k^{1-b}}=\frac{\Gamma^4(a+1)}{4a-b}.
\label{PANSR9}
\end{equation}
Hence, \eqref{PANSR8} together with \eqref{PANSR9} immediately imply \eqref{PANSR7}.
Furthermore, let $(P_n(u))$ be the martingale defined, for all $n \geq 1$, by
$$
P_n(u)=\sum_{k=1}^n k^{2a-b}a_k^2 \Big(\varepsilon_k^2(u) - \dE[\varepsilon_k^2(u) | \cF_{k-1}]\Big).
$$
Its predictable quadratic variation is given by
$$
\langle P \rangle_n = \sum_{k=1}^n k^{2(2a-b)} a_k^4
\Big(\dE[\varepsilon_k^4(u)| \mathcal{F}_{k-1}]- \dE^2[\varepsilon_k^2(u) | \cF_{k-1}]\Big).
$$
Consequently, we deduce from \eqref{UBMOM4EPS} and the Cauchy-Schwarz inequality that $(P_n(u))$ is bounded in $\dL^2$ as
\begin{equation*}
\sup_{n \geq 1} \dE\big[\langle P(u) \rangle_n \big] < \infty.
\end{equation*}
Therefore, it follows from Doob's martingale convergence theorem \cite{Hall1980}
that $(P_n(u))$ converges a.s. Finally, all the conditions of Theorem 1 and Corollaries 1 and 2 in \cite{Heyde1977} are satisfied, which leads, for any vector $u \in \dR^d$, to
\begin{equation}
\label{PANSR10}
\frac{M_n(u)-M(u)}{\sqrt{\Lambda_n(u)}} 
\underset{n\rightarrow+\infty}{\overset{\cL}{\longrightarrow}}\cN(0, 1).
\end{equation}
Hence, as $M_n(u)=a_n \langle u, S_n\rangle$ and $M(u)=\Gamma(a+1)\langle u, L\rangle$, we find from  \eqref{LIMAN}, \eqref{LIMSIGMAN2}, \eqref{PANSR6}  and \eqref{PANSR10} that
\begin{equation}
\label{PANSR11}
\frac{\langle u, S_n -n^a L\rangle}{\sqrt{\sigma^2_n}} \underset{n\rightarrow+\infty}{\overset{\cL}{\longrightarrow}} \cN\big(0, \vartheta^2 \|u\|^2\big).
\end{equation}
Consequently, we deduce the Gaussian fluctuation \eqref{ANSR1} from \eqref{PANSR11} together with the Cram\'er-Wold theorem.
Moreover, we also obtain from Theorem 1 in \cite{Heyde1977} that for any vector $u \in \dR^d$,
\begin{equation}
\label{PANSR12}
\sqrt{n^{2a-b}} \big( M_n(u)-M(u) \big) \underset{n\rightarrow+\infty}{\overset{\cL}{\longrightarrow}} \Gamma(a+1) \sqrt{\Sigma^\prime}\cN\big(0, \vartheta^2 \|u\|^2\big)
\end{equation}
where $\Sigma^\prime$ is independent of the Gaussian distribution at the right-hand side of \eqref{PANSR12}
and $\Sigma^\prime$ has a $\cM\cL(b)$ distribution. Finally, it follows from \eqref{LIMAN} and \eqref{PANSR12} that
\begin{equation*}
\frac{\langle u, S_n -n^a L\rangle}{\sqrt{n^b}} \underset{n\rightarrow+\infty}{\overset{\cL}{\longrightarrow}} \sqrt{\Sigma^\prime}\cN\big(0, \vartheta^2 \|u\|^2\big)
\end{equation*}
which leads to \eqref{ANSR2} and completes the proof of Theorem \ref{T-AN-SR}.
\demend

\bibliographystyle{abbrv}
\bibliography{Biblio-MERWS}

\end{document}